\documentclass[reqno]{amsart}
\usepackage{amssymb}
\usepackage{amsmath}
\usepackage{hyperref}
\usepackage{latexsym,amssymb,amsmath,amsfonts,amsthm,url,bbm}
\usepackage[comma,sort&compress]{natbib}

\allowdisplaybreaks

\newtheorem{theorem}{Theorem}[section]
\newtheorem{lemma}{Lemma}[section]
\newtheorem{proposition}{Proposition}[section]
\newtheorem{corollary}{Corollary}[section]

\theoremstyle{definition}

\theoremstyle{remark}
\newtheorem{remark}{Remark}[section]

\numberwithin{equation}{section}

\DeclareMathOperator{\E}{E}

\DeclareMathOperator{\prob}{P}

\DeclareMathOperator{\lito}{o}
\newcommand{\comment}[1]{}

\newcommand{\xinf}{X_{(\infty)}}
\newcommand{\xupk}{\sum_{t=m+1}^\infty \Theta_t X_t}
\newcommand{\xdnk}{\sum_{t=1}^m \Theta_t X_t}
\newcommand{\xtheta}{\Theta_tX_t}

\begin{document}
\bibliographystyle{natbib}
\title{Tail Behavior of Randomly Weighted Sums}

\author[R. S. Hazra]{Rajat Subhra Hazra}
\address{Rajat Subhra Hazra\\Statistics and Mathematics Unit\\
Indian Statistical Institute\\ 203 B.T.~Road\\ Kolkata 700108\\
India} \email{rajat\_r@isical.ac.in}

\author[K. Maulik]{Krishanu Maulik}
\address{Krishanu Maulik\\ Statistics and Mathematics Unit\\ Indian
Statistical Institute\\ 203 B.T.~Road\\ Kolkata 700108\\ India}
\email{krishanu@isical.ac.in}



\begin{abstract}
Let $\{X_t, t \geq 1\}$ be a sequence of identically distributed and pairwise asymptotically independent random variables with regularly varying tails and $\{ \Theta_t, t\geq1 \}$ be a sequence of positive random variables independent of the sequence $\{X_t, t \geq 1\}$. We shall discuss the tail probabilities and almost sure convergence of $\xinf=\sum_{t=1}^{\infty}\Theta_t X_t^{+}$ (where $X^+=\max\{0,X\}$) and $\max_{1\leq k<\infty} \sum_{t=1}^{k}\Theta_t X_t$ and provide some sufficient conditions motivated by~\cite{denisov:zwart:2007} as alternatives to the usual moment conditions. In particular, we illustrate how the conditions on the slowly varying function involved in the tail probability of $X_1$ helps to control the tail behavior of the randomly weighted sums.

Note that, the above results allow us to choose $X_1, X_2, \ldots$ as independent and identically distributed positive random variables. If $X_1$ has regularly varying tail of index $-\alpha$, where $\alpha>0$, and if $\{\Theta_t,t\geq1\}$ is a positive sequence of random variables independent of $\{X_t\}$, then it is known, which can also be obtained from the sufficient conditions above, that under some appropriate moment conditions on $\{\Theta_t,t\geq1\}$, $\xinf=\sum_{t=1}^{\infty}\Theta_t X_t$ converges with probability 1 and has regularly varying tail of index $-\alpha$. Motivated by the converse problems in \cite{JMRS:2009} we ask the question that, if $\xinf$ has regularly varying tail, then does $X_1$ have regularly varying tail under some appropriate conditions? We obtain appropriate sufficient moment conditions, including nonvanishing Mellin transform of $\sum_{t=1}^\infty \Theta_t$ along some vertical line in the complex plane, so that the above is true. We also show that the condition on the Mellin transform cannot be dropped.
\end{abstract}
\maketitle

\begin{section}{Introduction}\label{section:intro}
Let $\{X_t, t \geq 1\}$ be a sequence of identically distributed, pairwise asymptotically independent, cf.~\eqref{asymptotic independence}, random variables and $\{ \Theta_t, t\geq 1 \}$ be a sequence of positive random variables independent of the sequence $\{X_t, t \geq 1\}$. We shall discuss the tail probabilities and almost sure convergence of $\xinf=\sum_{t=1}^{\infty}\Theta_t X_t^{+}$ (where $X^+=\max\{0,X\}$)  and $\max_{1\leq k<\infty}\sum_{t=1}^{k}\Theta_t X_t$, in particular, when $X_t$'s belong to the class of random variables with regularly varying tail and $\{\Theta_t, t\geq 1\}$ satisfies some moment conditions. We shall say that a random variable $X$ with distribution function $F$ has \textit{regularly varying tail of index $-\alpha$}, if $\overline F(x) := 1-F(x)$ is a regularly varying function of index $-\alpha$, that is, for any $t>0$, as $x\to\infty$, $\overline F(tx) \sim t^{-\alpha} \overline F(x)$. Here and later, for two positive functions $a(x)$ and $b(x)$, we write $a(x) \sim b(x)$ as $x\to\infty$, if $\lim_{x\to\infty} a(x)/b(x) = 1$. For $\alpha>0$, the convergence in the limit of the ratio of the tail probabilities is uniform in $t$ on the intervals of the form $[a,\infty)$ with $a>0$. Note that, we require the upper endpoint of the support of $X$ to be $\infty$. In recent times, there have been quite a few articles devoted to the asymptotic tail behavior of randomly weighted sums and their maxima,  \citep[see, for example,][]{resnick:willekens:1991, tang:chen:ng:2005, tang:wang:2006, zhang:weng:shen:2008, hult:somorodnitsky:2008}.

The question about the tail behavior of the infinite series $\xinf$ with non-random $\Theta_t$ and i.i.d.\ $X_t$ having regularly varying tails has been studied well in the literature, as it arises in the context of the linear processes, including ARMA and FARIMA processes. We refer to \cite{jessen:mikosch:2006} for a review of the results. The case, when $\Theta_t$'s are random, arises in various areas, especially in actuarial and economic situations and stochastic recurrence equation. For various applications, see \cite{hult:somorodnitsky:2008, zhang:weng:shen:2008}.

\cite{resnick:willekens:1991} showed that if $\{X_t\}$ is a sequence of i.i.d.\ nonnegative random variables with regularly varying tail of index $-\alpha$, where $\alpha>0$ and $\{\Theta_t\}$ is another sequence of positive random variables independent of $\{X_t\}$, the series $\xinf$ has regularly varying tail under the following conditions, which we shall call the RW conditions:
\renewcommand{\theenumi}{RW\arabic{enumi}}
\begin{enumerate}
\item \label{RW1} If $0 <\alpha<1$, then for some $\epsilon\in(0,\alpha)$, $\sum_{t=1}^{\infty} \E[\Theta_t^{\alpha+\epsilon} + \Theta_t^{\alpha-\epsilon}]<\infty$.
\item \label{RW2} If $1\leq\alpha<\infty$, then for some $\epsilon\in(0,\alpha)$, $\sum_{t=1}^{\infty} (\E[\Theta_t^{\alpha+\epsilon} + \Theta_t^{\alpha-\epsilon}])^{\frac{1}{\alpha+\epsilon}}<\infty$.
\end{enumerate}
In this case, we have $\prob[\xinf>x] \sim \sum_{t=1}^\infty \E[\Theta_t^\alpha] \prob[X_1>x]$ as $x\to\infty$.
\begin{remark} \label{rem: RW}
Each of the RW conditions implies the other for the respective ranges of $\alpha$. In particular, if $0<\alpha<1$, choose $\epsilon^\prime<\epsilon$ such that $\alpha+\epsilon^\prime<1$. Note that
\begin{multline*}
\sum_{t=1}^{\infty} \E[\Theta_t^{\alpha+\epsilon^\prime} + \Theta_t^{\alpha-\epsilon^\prime}] \le 2 \sum_{t=1}^{\infty} \E[\Theta_t^{\alpha+\epsilon^\prime} \mathbbm 1_{[\Theta_t\ge 1]} + \Theta_t^{\alpha-\epsilon^\prime} \mathbbm 1_{[\Theta_t< 1]}]\\
\le 2 \sum_{t=1}^{\infty} \E[\Theta_t^{\alpha+\epsilon} \mathbbm 1_{[\Theta_t\ge 1]} + \Theta_t^{\alpha-\epsilon} \mathbbm 1_{[\Theta_t< 1]}] \le 2\sum_{t=1}^{\infty} \E[\Theta_t^{\alpha+\epsilon} + \Theta_t^{\alpha-\epsilon}]<\infty.
\end{multline*}
Further, since $\alpha+\epsilon^\prime<1$, we also have $\sum_{t=1}^{\infty} (\E[\Theta_t^{\alpha+\epsilon^\prime} + \Theta_t^{\alpha-\epsilon^\prime}])^{\frac{1}{\alpha+\epsilon^\prime}}<\infty$. On the other hand, if $\alpha\ge 1$ and $\epsilon>0$, then $\alpha+\epsilon>1$ and the condition~\eqref{RW2} implies $\sum_{t=1}^{\infty} \E[\Theta_t^{\alpha+\epsilon} + \Theta_t^{\alpha-\epsilon}]<\infty$.
\end{remark}

\cite{zhang:weng:shen:2008} considered the tails of $\sum_{t=1}^n \Theta_tX_t$ and the tails of their maxima, when $\{X_t\}$ are pairwise asymptotically independent and have extended regularly varying and negligible left tail and $\{\Theta_t\}$ are positive random variables independent of $\{X_t\}$. The sufficient conditions for the tails to be regularly varying are almost similar.

While the tail behavior of $\xinf$ requires only the $\alpha$-th moments of $\Theta_t$'s, we require existence and summability of some extra moments in the RW conditions. Note that $\Theta_t^{\alpha+\epsilon}$ acts as a dominator for $[\Theta_t\ge1]$ and $\Theta_t^{\alpha-\epsilon}$ acts as a dominator for $[\Theta_t\le1]$. In some cases, the assumption of existence and summability of extra moments can become restrictive. For example, consider $\{\Theta_t\}$ such that $\sum_{t=1}^{\infty}\E[\Theta_t^{\alpha+\epsilon}]=\infty$ for all $\epsilon>0$ but $\sum_{t=1}^{\infty} \E[\Theta_t^{\alpha}]<\infty$. (A particular choice of such $\{\Theta_t\}$, for $\alpha<1$ is as follows: $\Theta_t$ takes values ${2^t}/{t^{2/\alpha}}$ and $0$ with probability ${2^{-t\alpha}}$ and $1-2^{-t\alpha}$ respectively.) Also let  $\{X_t\}$ be i.i.d.\ Pareto with parameter $\alpha<1$, independent of $\{\Theta_t\}$. Then it turns out, after some easy calculations, that $\sum_{t=1}^{\infty}\Theta_t X_t$ is finite almost surely and has regularly varying tail of index $-\alpha$. This leads to the question whether we can reduce the moment conditions on $\Theta_t$ to obtain the regular variation of the tail for $\xinf$.

The situation becomes clearer when we consider a general term of the series. It involves the product $\Theta_t X_t$. Using Breiman's theorem \citep[cf.][]{breiman:1965, cline:somorodnitsky:1994}, the tail behavior of the product depends on the moments of $\Theta_t$. Breiman's theorem states that, if $X$ is a random variable with regularly varying tail of index $-\alpha$ for some $\alpha> 0$ and is independent of a positive random variable $\Theta$ satisfying $\E[\Theta^{\alpha+\epsilon}]<\infty$ for some $\epsilon>0$, then,
\begin{equation}\label{breiman}
\lim_{x\to\infty} \prob[\Theta X > x]\sim \E[\Theta^{\alpha}] \prob[X>x].
\end{equation}
Note that, in this case, we work with a probability measure $\prob [\Theta_t \in \cdot]$, unlike in the problem of the weighted sum, where a $\sigma$-finite measure $\sum_{t=1}^\infty \prob [\Theta\in\cdot]$ is considered. In this case, we can consider the dominator as $1$ on $[\Theta\le1]$ instead of $\Theta^{\alpha-\epsilon}$, since $1$ is integrable with respect to a probability measure.

\cite{denisov:zwart:2007} relaxed the existence of $(\alpha+\epsilon)$ moments in Breiman's theorem to $\E[\Theta^{\alpha}] < \infty$. They also made the further natural assumption that $\prob[\Theta>x]=\lito(\prob[X>x])$. However, to obtain~\eqref{breiman}, the weaker moment assumption needed to be compensated. They obtained some sufficient conditions for~\eqref{breiman} to hold. We would like to find conditions similar to those obtained by \cite{denisov:zwart:2007}, which will guarantee the regular variation of $\xinf$.

In the above discussion, we considered the effect of the tail of $X_1$ in determining the tail of $\xinf$. However, the converse question is also equally interesting. More specifically, let $\{X_t\}$ be a sequence of identically distributed, asymptotically independent, positive random variables, independent of another sequence of positive random variables $\{\Theta_t\}$. As before, we define $\xinf=\sum_{t=1}^{\infty}\Theta_t X_t$ and assume that $\xinf$ converges with probability one and has regularly varying tail of index $-\alpha$ with $\alpha>0$. It is interesting to obtain sufficient conditions which will ensure that $X_1$ has a regularly varying tail of index $-\alpha$ as well.

Similar converse questions regarding Breiman's theorem~\eqref{breiman} have recently been considered in the literature. Suppose $X$ and $Y$ are positive random variables with $\E[Y^{\alpha+\epsilon}]<\infty$ and the product $XY$ has regularly varying tail of index $-\alpha$, $\alpha>0$. Then it was shown in \cite{JMRS:2009} that $X$ has regularly varying tail of same index and hence~\eqref{breiman} holds. They have also obtained results for the weighted series, when the weights $\{\Theta_t\}$ are nonrandom. We shall extend this result for product to the case of randomly weighted series under appropriate conditions.

In Section~\ref{section:notation} we first describe the various classes of heavy tailed distributions and describe the conditions imposed by \cite{denisov:zwart:2007}. We study the tail behavior when finite weighted sums are considered. In Section~\ref{sec:direct} we describe the tail behavior of the series of randomly weighted sums. In Section~\ref{sec: converse} we consider the converse problem described above. We prove the converse result is true under the RW conditions and the extra assumption of nonvanishing Mellin transform. We also show the necessity of this extra assumption.
\end{section}

\begin{section}{Notations and preliminary results} \label{section:notation}
We first introduce a few classes of random variables, which will be required for the rest of the discussion. A random variable $X$ with distribution function $F$ is called \textit{long tailed}, if for any fixed $y\in\mathbb{R}$ and as $x \rightarrow \infty $, we have $\overline{F}(x-y) \sim \overline{F}(x)$. The class of long tailed distribution is denoted by $\mathcal{L}$. Observe that for $F\in\mathcal L$, we need $\overline F(x)>0$ for all $x>0$. The class $\mathcal{L}$ is related to the class of distributions with regularly varying tail by the fact that $F\in \mathcal{L}$ if and only if $\overline{F}(\log(\cdot))$ is slowly varying, that is, regularly varying of index $0$. Equivalently, the random variable $X$ has distribution function in the class $\mathcal{L}$ if and only if $\exp(X)$ has a regularly varying tail of index $0$.

A nonnegative function $f$, which does not vanish for all large $x$, is in the class of \textit{subexponential densities} (denoted by $\mathcal{S}_{d}$), if it satisfies the property
$$\lim_{x\rightarrow\infty}\int_{0}^{x}\frac{f(x-y)}{f(x)}f(y) dy = 2\int_{0}^{\infty} f(u)du < \infty.$$
If, for some random variable $X$, the tail probability distribution function $\overline{F}(x) = \prob[ X > x]$ is in the class of subexponential densities, we say that $X$ is in $\mathcal{S^{*}}$. Again, if $F\in\mathcal S^*$, we need $\overline F(x)>0$ for all $x>0$.

A distribution function $F$ belongs to the class $\mathcal{S(\gamma)}$ with $\gamma\geq 0$ if, for all real $u$,
$$\lim_{x\to\infty} \frac{\overline F(x-u)}{\overline F(x)} = e^{\gamma u} \quad\text{and}\quad \lim_{x\rightarrow\infty} \frac{\overline{F*F}(x)}{\overline{F}(x)} = 2 \int^{\infty}_{0} e^{\gamma y} F(dy) < \infty.$$
The class $\mathcal{S} := \mathcal{S}$(0) is called the class of \textit{subexponential distribution  functions}. See \cite{kluppelberg:1988, kluppelberg:1989, embrechets:kluppelberg:mikosch:1997} for properties of these classes.

We call two random variables $X_1$ and $X_2$ to be \textit{asymptotically independent} if
\begin{equation} \label{asymptotic independence}
\lim_{x\rightarrow\infty}\frac{\prob[X_1 > x,X_2 > x]}{\prob[X_t > x]} = 0, \text{ for $t=1, 2$.}
\end{equation}
See \cite{ledford:tawn:1996, ledford:tawn:1997} or Chapter~6.5 of \cite{resnick:2007} for discussions on asymptotic independence. Note that, we require $\overline F_t(x)>0$ for all $x>0$ and $t=1,2$. Observe that if $X_1$ and $X_2$ are independent, then they are also asymptotically independent. Thus the results under pairwise asymptotic independence condition continue to hold in the independent setup.

A random variable $X$ is said to have \textit{negligible left tail with respect to the right one}, if
\begin{equation} \label{right tail}
\lim_{x\rightarrow\infty}\frac{\prob[X<-x]}{\prob[X > x]}=0.
\end{equation}
Note that we require $\prob[X>x]>0$ for all $x>0$.

The random variables with regularly varying tails will play a central role in this article. Note that, if $X$ has a regularly varying tail of index $-\alpha$, then $x^\alpha \prob[X>x]$ is a slowly varying function, that is, a regularly varying function with index $0$. By Karamata's representation, a slowly varying function $L$ can be one of the four types \citep[cf.][Lemma 2.1] {denisov:zwart:2007}, namely,
\begin{enumerate}
\item \label{L1}$L(x) = c(x)$,
\item \label{L2}$L(x) = c(x)/\prob[ V > \log x]$,	
\item \label{L3}$L(x) = c(x)\prob[ U > \log x]$,
\item \label{L4}$L(x) = c(x)\prob[ U > \log x]/\prob[ V > \log x]$.
\end{enumerate}
In the above representations, $c(x)$ is a function converging to $c\in(0,\infty)$, and $U$ and $V$ are two long-tailed random variables with hazard rates converging to $0$. We shall refer to a slowly varying function $L$ as of type 1, type 2, type 3 or type 4, according to the above representations.

\cite{denisov:zwart:2007} introduced the following sufficient conditions on the slowly varying part $L$ of the regularly varying tail of index $-\alpha$ of a random variable $X$ with distribution function $\overline F(x)=x^{-\alpha} L(x)$ for Breiman's theorem~\eqref{breiman} to hold:
\let\myenumi\theenumi
\renewcommand{\theenumi}{DZ\arabic{enumi}}
\begin{enumerate}
   \item \label{DZ1} Assume $\lim_{x\rightarrow\infty}\sup_{y\in[1,x]} {L(y)}/{L(x)}:=D_1 < \infty$.
   \item \label{DZ2} Assume $L$ is of type 3 or type 4 and  $L(e^{x}) \in \mathcal{S}_{d}$.
   \item \label{DZ3} Assume $L$ is of type 3 or type 4, $U\in \mathcal{S^{*}}$ and $\prob[\Theta > x] = \lito(x^{-\alpha}\prob[U>\log x])$.
   \item \label{DZ4} When $\E[U]=\infty$ or equivalently $\E[X^\alpha]=\infty$, define $m(x)=\int_{0}^{x} v^{\alpha} F(dv) \to \infty$. Assume $\limsup_{x\rightarrow\infty} \sup_{\sqrt{x}\leq y \leq x} {L(y)}/{L(x)} :=D_2 < \infty$ and $\prob[\Theta > x] = \lito({\prob[X>x]}/m(x))$.
\end{enumerate}
\renewcommand{\theenumi}{\myenumi}
We shall refer to these conditions as the DZ conditions. For further discussions on the DZ conditions, we refer to \cite{denisov:zwart:2007}. Denisov and Zwart proved the following lemma:
\begin{lemma}[\citealp{denisov:zwart:2007}, Section~2] \label{breiman new}
Let $X$ be a nonnegative random variable with regularly varying tail of index $-\alpha$, $\alpha \geq 0$ and $\Theta$ be a positive random variable independent of $X$ such that $\E[\Theta^{\alpha}]<\infty$ and $\prob[\Theta > x]= \lito(\prob[X>x])$. If $X$ and ${\Theta}$ satisfy any one of the DZ conditions, then~\eqref{breiman} holds.
\end{lemma}

The next result shows that asymptotic independence is preserved under multiplication, when the DZ conditions are assumed.
\begin{lemma} \label{joint asymptotic independence}
Let $X_1, X_2$ be two positive, asymptotically independent, identically distributed random variables with common regularly varying tail of index $-\alpha$, where $\alpha>0$. Let $\Theta_1$ and $\Theta_2$ be two other positive random variables independent of the pair $(X_1, X_2)$ satisfying $\E[\Theta_t^{\alpha}] < \infty$,  $t=1,2$. Also suppose that $\prob[\Theta_t > x]= \lito(\prob[X_1>x])$ for $t=1,2$ and the pairs $(\Theta_1, X_1)$ and $(\Theta_2, X_2)$ satisfy any one of the DZ conditions. Then $\Theta_1 X_1$ and $\Theta_2 X_2$ are asymptotically independent.
\end{lemma}
\begin{proof}
Here and later $G$ will denote the joint distribution function of $(\Theta_1,\Theta_2)$ and $G_t$ will denote the marginal distribution functions of $\Theta_t$.
\begin{multline*}
\frac{\prob[\Theta_1 X_1 > x,\Theta_2X_2 > x]}{\prob[X_1>x]} = \iint_{u\leq v}+\iint_{u>v} \frac{\prob[X_1>x/u,X_2>x/v]}{\prob[X_1>x]} G(du,dv)\\
\leq \int_0^\infty \frac{\prob [X_1>x/v, X_2>x/v]}{\prob[X_1>x/v]} \frac{\prob[X_1>x/v]}{\prob[X_1>x]} (G_1+G_2)(dv).
\end{multline*}
The integrand converges to $0$. Also, the first factor of the integrand is bounded by $1$ and hence the integrand is bounded by the second factor, which converges to $v^\alpha$. Further, using Lemma~\ref{breiman new}, we have
\begin{multline*}
\int_0^\infty \frac{\prob[X_1>x/v]}{\prob[X_1>x]} (G_1+G_2)(dv) = \frac{\prob[\Theta_1X_1>x] + \prob[\Theta_2X_1>x]}{\prob[X_1>x]}\\
\to \E\left[ \Theta_1^\alpha \right] + \E\left[ \Theta_2^\alpha \right] = \int_0^\infty v^\alpha (G_1+G_2)(dv).
\end{multline*}
Then the result follows using Pratt's lemma, cf. \cite{pratt:1960}.
\end{proof}

The next lemma shows that if the left tail of $X$ is negligible when compared to the right tail then the product has also such a behavior.
\begin{lemma} \label{left tail negligibility}
Let $X$ have regularly varying tail of index $-\alpha$, for some $\alpha>0$ satisfying~\eqref{right tail} and $\Theta$ be independent of $X$ satisfying $\E[\Theta^{\alpha}]<\infty$ and $\prob[\Theta>x]= \lito(\prob[X > x])$. Also suppose that $(\Theta,X)$ satisfy one of the DZ conditions. Then, for any $u>0$,
$$\lim_{x\rightarrow\infty}\frac{\prob[\Theta X < -ux]}{\prob[\Theta X > x]}=0.$$
\end{lemma}
The proof is exactly similar to that of Lemma~\ref{joint asymptotic independence}, except for the fact that the first factor in the integrand is bounded, as, using~\eqref{right tail}, $\prob[X<-x]/\prob[X>x]$ is bounded. We skip the proof.

The following result from \cite{davis:resnick:1996} considers a simple case of tail of sum of finitely many random variables.
\begin{lemma}[\citealp{davis:resnick:1996}, Lemma~2.1] \label{finite sum asymptotics1}
Suppose $Y_1, Y_2,\ldots,Y_k$ are nonnegative, pairwise asymptotically independent $($but not necessarily identically distributed$)$ random variables with regularly varying tails of common index $-\alpha$, where $\alpha>0$. If, for $t=1, 2, \ldots, k$,
${\prob[Y_t>x]}/\prob[Y_1>x]\rightarrow c_t$, then
$$\frac{\prob[\sum_{t=1}^k Y_t>x]}{\prob[Y_1>x]}\rightarrow c_1+c_2+\cdots+c_k.$$
\end{lemma}

We have the following corollary by applying Lemma~\ref{finite sum asymptotics1} with $Y_t=\Theta_tX_t^+$ and the modified Breiman's theorem in Lemma~\ref{breiman new} under the DZ conditions.
\begin{corollary} \label{fin sum cor}
Let $\{X_t\}$ be a sequence of pairwise asymptotically independent, identically distributed random variables with common regularly varying tail of index $-\alpha$, where $\alpha>0$, which is independent of another sequence of positive random variables $\{\Theta_t\}$ satisfying $\E[\Theta_t^\alpha]<\infty$, for all $t$. Also assume that, for all $t$, $\prob [\Theta_t > x] = \lito(\prob[X_1>x])$ and the pairs $(\Theta_t, X_t)$ satisfy one of the DZ conditions. Then we have
$$\prob\left[\sum_{t=1}^k\Theta_tX_t^+>x\right]\sim \prob[X_1>x]\sum_{t=1}^k\E[\Theta_t^{\alpha}].$$
\end{corollary}

Using Lemmas~\ref{breiman new}--\ref{finite sum asymptotics1} and Corollary~\ref{fin sum cor} and arguing as in Theorem 3.1(a) of \cite{zhang:weng:shen:2008}, we have the following result. (Note that the proof of Theorem~3.1(a) of \cite{zhang:weng:shen:2008} require only the results obtained in Lemmas~\ref{breiman new}--\ref{finite sum asymptotics1} and Corollary~\ref{fin sum cor}.)
\begin{proposition} \label{finite sum}
Let $\{X_t\}$ be a sequence of pairwise asymptotically independent, identically distributed random variables with common regularly varying tail of index $-\alpha$, for some $\alpha>0$ satisfying~\eqref{right tail}, which is independent of another sequence of positive random variables $\{\Theta_t\}$. Further assume that, for all $t$, $\prob[\Theta_t>x]= \lito(\prob[X_1>x])$ and $\E[\Theta_t^{\alpha}]<\infty$. Also assume that the pairs $(\Theta_t, X_t)$ satisfy one of the DZ conditions. Then,
$$\prob\left[\max_{1\leq k \leq n}\sum_{t=1}^{k}\Theta_t X_t > x\right] \sim \prob\left[\sum_{t=1}^{n}\Theta_t X_t^{+}> x\right] \sim \prob[X_1>x] \sum_{t=1}^{n}\E[\Theta_t^{\alpha}].$$
\end{proposition}
\end{section}

\begin{section}{The tail of the weighted sum under the DZ conditions} \label{sec:direct}
In Proposition~\ref{finite sum}, we saw that the conditions on the slowly varying function helps us to reduce the moment conditions on $\{\Theta_t\}$ for the finite sum. However we need some additional hypotheses to handle the infinite series. To study the almost sure convergence of $\xinf=\sum_{t=1}^{\infty} \Theta_t X_t^{+}$, observe that the partial sums $S_n=\sum_{t=1}^{n}\Theta_t X_t^{+}$ increase to $\xinf$. We shall show in the following results that $\prob[\xinf>x]\sim\prob[X_1>x]\sum_{t=1}^{\infty}\E[\Theta_t^{\alpha}]$ under suitable conditions. Thus if $\sum_{t=1}^{\infty}\E[\Theta_t^{\alpha}] < \infty$, then $\lim_{x\rightarrow\infty}\prob[\xinf>x]=0$ and $\xinf$ is finite almost surely.

To obtain the required tail behavior, we shall assume the following conditions, which weaken the moment requirements of $\{\Theta_t\}$ assumed in the conditions~\eqref{RW1} and~\eqref{RW2} given in \cite{resnick:willekens:1991}:
\renewcommand{\theenumi}{RW\arabic{enumi}$^\prime$}
\begin{enumerate}
    \item \label{ERW1} For $ 0 <\alpha<1$, $\sum_{t=1}^{\infty} \E[\Theta_t^{\alpha}]<\infty$.
    \item \label{ERW2} For $1\leq\alpha<\infty$, for some $\epsilon>0$, $\sum_{t=1}^{\infty} (\E[\Theta_t^{\alpha}])^{\frac{1}{\alpha+\epsilon}} <\infty$.
\end{enumerate}
\renewcommand{\theenumi}{\myenumi}
We shall call these conditions modified RW moment conditions.

\begin{remark} \label{almost sure remark}
As observed in Remark~\ref{rem: RW}, for $\alpha\geq 1$ and $\epsilon>0$, the finiteness of the sum in~\eqref{ERW2} implies $\sum_{t=1}^{\infty}(\E[\Theta_t^{\alpha}])<\infty$. Thus to check the almost sure finiteness of $\xinf$, it is enough to check the tail asymptotics condition:
$$\prob[\xinf > x]\sim \prob[X_1>x]\sum_{t=1}^{\infty}\E[\Theta_t^{\alpha}].$$

We shall prove it under the above model together with the assumption that $\prob[\Theta_t>x] = \lito(\prob[X_1>x])$ and one of the DZ conditions. We need to assume an extra summability condition for uniform convergence, when the conditions~\eqref{DZ2},~\eqref{DZ3} or~\eqref{DZ4} hold.

Further note that $\Theta_1 X_1 \le \max_{1\le n<\infty} \sum_{t=1}^n \Theta_t X_t \le X_{(\infty)}$ and hence the almost sure finiteness of $X_{(\infty)}$ guarantees that $\max_{1\le n<\infty} \sum_{t=1}^n \Theta_t X_t$ is a valid random variable.
\end{remark}

\begin{theorem}
Suppose that $\{X_t\}$ is a sequence of pairwise asymptotically independent, identically distributed random variables with common regularly varying tail of index $-\alpha$, where $\alpha>0$, satisfying~\eqref{right tail}, which is independent of another sequence of positive random variables $\{\Theta_t\}$. Also assume that $\prob[\Theta_t>x]=\lito(\prob[X_1>x])$, the pairs $(\Theta_t, X_t)$ satisfy one of the four DZ conditions~\eqref{DZ1}--\eqref{DZ4} and, depending on the value of $\alpha$, the modified RW moment conditions~\eqref{ERW1} or~\eqref{ERW2} holds. If the pairs $(\Theta_t, X_t)$ satisfy DZ condition~\eqref{DZ2},~\eqref{DZ3} or~\eqref{DZ4}, define
\begin{equation} \label{Ct}
C_t =
\begin{cases}
\sup_x \frac{\prob[\Theta_t > x]}{\prob[X_1>x]}, &\text{when~\eqref{DZ2} holds,}\\[1ex]
\sup_x \frac{\prob[\Theta_t > x]}{x^{-\alpha}\prob[U>\log x]}, &\text{when~\eqref{DZ3} holds,}\\[1ex]
\sup_x \frac{\prob[\Theta_t > x]}{\prob[X_1>x]} m(x), &\text{when~\eqref{DZ4} holds,}
\end{cases}
\end{equation}
and further assume that
\begin{align}
&\sum_{t=1}^\infty C_t < \infty, &\text{when $\alpha<1$,} \label{sum less}\\
&\sum_{t=1}^\infty C_t^{\frac1{\alpha+\epsilon}} < \infty, &\text{when $\alpha\ge 1$.} \label{sum more}
\end{align}

Then
$$\prob\left[\max_{1\leq n <\infty}\sum_{t=1}^{n}\Theta_t X_t > x\right]\sim \prob[\xinf > x]\sim \prob[X_1>x]\sum_{t=1}^{\infty}\E[\Theta_t^{\alpha}]$$
and $\xinf$ is almost surely finite.
\end{theorem}
\begin{proof}
For any $m\geq1$, we have, by Proposition~\ref{finite sum},
$$\prob\left[\max_{1\leq n <\infty}\sum_{t=1}^{n}\Theta_t X_t >x\right] \geq \prob\left[\max_{1\leq n\leq m}\sum_{t=1}^{n}\Theta_t X_t > x\right] \sim \prob[X_1>x]\sum_{t=1}^{m}\E[\Theta_t^{\alpha}]$$
leading to
$$\liminf_{x\rightarrow\infty}\frac{\prob[\max_{1\leq n <\infty}\sum_{t=1}^{n}\Theta_t X_t >x]}{\prob[X_1>x]} \geq \sum_{t=1}^{\infty}\E[\Theta_t^{\alpha}].$$
Similarly, comparing with the partial sums and using Proposition~\ref{finite sum}, we also get
$$\liminf_{x\rightarrow\infty}\frac{\prob[\xinf >x]}{\prob[X_1>x]} \geq \sum_{t=1}^{\infty}\E[\Theta_t^{\alpha}].$$

For the other inequality, observe that for any natural number $m$, $0<\delta<1$ and $x\geq 0$,
\begin{multline*}
\prob\left[ \max_{1\leq n <\infty}\sum_{t=1}^{n}\Theta_tX_t > x\right]\\
\leq \prob\left[\max_{1\leq n \leq m}\sum_{t=1}^{n}\Theta_tX_t>(1-\delta)x\right] + \prob\left[\sum_{t=m+1}^{\infty}\Theta_tX_t^{+}>\delta x\right].
\end{multline*}
For the first term, by Proposition~\ref{finite sum} and the regular variation of the tail of $X_1$, we have,
\begin{multline*}
\lim_{x\rightarrow\infty} \frac{\prob\left[\max_{1\leq n \leq m}\sum_{t=1}^{n}\Theta_tX_t>(1-\delta)x\right]}{\prob[X_1>x]}\\ = (1-\delta)^{-\alpha} \sum_{t=1}^m \E[\Theta_t^{\alpha}] \leq (1-\delta)^{-\alpha} \sum_{t=1}^{\infty}\E[\Theta_t^{\alpha}].
\end{multline*}
Also, for $\xinf$, we have,
$$\prob\left[ \xinf > x\right] \leq\ \prob\left[\sum_{t=1}^{m}\Theta_tX_t^+>(1-\delta)x\right] + \prob\left[\sum_{t=m+1}^{\infty}\Theta_tX_t^{+}>\delta x\right]$$
and a similar result holds for the first term.

Then, as $X_1$ is a random variable with regularly varying tail, to complete the proof, it is enough to show that,
\begin{equation} \label{tail negligible}
\lim_{m\rightarrow\infty}\limsup_{x\rightarrow\infty}\frac{\prob[\sum_{t=m+1}^{\infty}\Theta_tX_t^{+} >x]}{\prob[X_1>x]}= 0.
\end{equation}
Now,
\begin{align}
&\prob\left[\sum_{t=m+1}^{\infty}\Theta_t{X_t}^{+} > x\right] \nonumber\\
\leq &\prob\left[\bigvee\limits_{t=m+1}^{\infty}\Theta_t X_t^{+}>x\right] +\prob\left[\sum_{t=m+1}^{\infty}\Theta_tX_t^{+}>x, \bigvee\limits_{t=m+1}^{\infty}\Theta_t X_t^{+}\leq x\right] \nonumber\\
\leq &\sum_{t=m+1}^\infty \prob[\Theta_t X_t > x] + \prob\left[\sum_{t=m+1}^{\infty}\Theta_tX_t^{+}\mathbbm{1}_{[\Theta_t X_t^+ \leq x]}>x\right]. \label{two terms tail negligible}
\end{align}

We bound the final term of~\eqref{two terms tail negligible} separately in the cases $\alpha<1$ and $\alpha\ge 1$. In the rest of the proof, for $\alpha\ge 1$, we shall choose $\epsilon>0$, so that the condition~\eqref{ERW2} holds. We first consider the case $\alpha<1$. By Markov inequality, the final term of~\eqref{two terms tail negligible} gets bounded above by
\begin{align}
\sum_{t=m+1}^{\infty}\frac1x\E\left[\Theta_t X_t^+\mathbbm{1}_{[\Theta_t X_t^+ \leq x]}\right]
&=\sum_{t=m+1}^{\infty}\int_0^{\infty}\frac1{x/v}{\E\left[X_t^+\mathbbm{1}_{[X_t^+\leq x/v]}\right]}G_t(dv) \label{eq: term 2 alpha small}\\
&=\sum_{t=m+1}^{\infty}\int_0^{\infty}\frac{\E[X_t^+\mathbbm{1}_{[X_t^+\leq x/v]}]}{x/v\prob[X_t^+>x/v]}\prob[X_t>x/v]G_t(dv).\nonumber
\end{align}
Now, using Karamata's theorem \citep[cf.][Theorem~2.1]{resnick:2007}, we have
$$\lim_{x\to\infty} \frac{\E\left[X_t^+\mathbbm{1}_{[X_t^+\leq x/v]}\right]}{x\prob[X_t^+>x]} = \frac\alpha{1-\alpha}$$
and, for $x<1$, we have $\E[X_t^+\mathbbm{1}_{[X_t^+\leq x/v]}]/(x\prob[X_t^+>x])\le 1/\prob[X_t^+>1]$. Thus, $\E[X_t^+\mathbbm{1}_{[X_t^+\leq x/v]}]/(x\prob[X_t^+>x])$ is bounded on $(0,\infty)$. So the final term of~\eqref{two terms tail negligible} becomes bounded by a multiple of $\sum_{t=m+1}^\infty \prob[\Theta_t X_t>x]$.

When $\alpha\geq 1$, using Markov inequality on the final term of~\eqref{two terms tail negligible}, we get a bound for it as
\begin{align}
&\frac1{x^{\alpha+\epsilon}} \E\left[ \left( \sum_{t=m+1}^\infty \Theta_t X_t^+ \mathbbm{1}_{[\Theta_t X_t^+\le x]} \right)^{\alpha+\epsilon} \right],\nonumber
\intertext{and then using Minkowski's inequality, this gets further bounded by}
&\left\{ \sum_{t=m+1}^\infty \left( \E \left[ \frac1{x^{\alpha+\epsilon}} \left(\Theta_t X_t^+\right)^{\alpha+\epsilon} \mathbbm{1}_{[\Theta_t X_t^+\le x]} \right] \right)^{\frac1{\alpha+\epsilon}} \right\}^{\alpha+\epsilon}\nonumber\\
= &\left\{ \sum_{t=m+1}^\infty \left[ \int_0^\infty (x/v)^{-(\alpha+\epsilon)} {\E \left[ \left(X_t^+\right)^{\alpha+\epsilon} \mathbbm{1}_{[X_t^+\le x/v]} \right]} G_t(dv) \right]^{\frac1{\alpha+\epsilon}} \right\}^{\alpha+\epsilon} \label{eq: term 2 alpha large}\\
= &\left\{ \sum_{t=m+1}^\infty \left[ \int_0^\infty \frac{\E \left[ \left(X_t^+\right)^{\alpha+\epsilon} \mathbbm{1}_{[X_t^+\le x/v]} \right]}{(x/v)^{\alpha+\epsilon} \prob[X_t^+ > x/v]} \prob[X_t>x/v] G_t(dv) \right]^{\frac1{\alpha+\epsilon}} \right\}^{\alpha+\epsilon}.\nonumber
\end{align}
Then, again using Karamata's theorem, the first factor of the integrand converges to $\alpha/\epsilon$ and, arguing as in the case $\alpha<1$, is bounded. Thus the final term of~\eqref{two terms tail negligible} is bounded by a multiple of $[ \sum_{t=m+1}^\infty ( \prob[\Theta_t X_t > x] )^{1/(\alpha+\epsilon)} ]^{\alpha+\epsilon}$.

Combining the two cases for $\alpha$, we get, for some $L_1>0$,
$$\frac{\prob[\sum_{t=m+1}^{\infty}\Theta_tX_t^{+} >x]}{\prob[X_1>x]} \le
\begin{cases}
L_1 \sum_{t=m+1}^\infty \frac{\prob[\Theta_t X_t>x]}{\prob[X_1>x]}, &\text{when $\alpha<1$,}\\[2ex]
\sum_{t=m+1}^\infty \frac{\prob[\Theta_t X_t>x]}{\prob[X_1>x]}\\
 + L_1 \left[ \sum_{t=m+1}^\infty \left( \frac{\prob[\Theta_t X_t > x]}{\prob[X_1>x]} \right)^{\frac1{\alpha+\epsilon}} \right]^{\alpha+\epsilon}, &\text{when $\alpha\ge 1$.}
\end{cases}$$
To prove~\eqref{tail negligible}, we shall show
\begin{equation} \label{eq: dir bd}
\frac{\prob[\Theta_t X_t > x]}{\prob[X_1>x]} \le B_t
\end{equation}
for all large values of $x$, where
\begin{equation} \label{eq: bd sum}
\begin{split}
\sum_{t=1}^\infty B_t <\infty, &\text{ for $\alpha<1$,}\\
\sum_{t=1}^\infty B_t^{\frac1{\alpha+\epsilon}} < \infty, &\text{ for $\alpha\ge1$.}
\end{split}
\end{equation}
As mentioned in Remark~\ref{almost sure remark}, for $\alpha\ge 1$ and $\epsilon>0$, $\sum_{t=1}^\infty B_t^{1/{\alpha+\epsilon}} < \infty$ will also imply $\sum_{t=1}^\infty B_t <\infty$. Thus, for both the cases of $\alpha<1$ and $\alpha\ge1$, the sums involved will be bounded by the tail sum of a convergent series and hence~\eqref{tail negligible} will hold.

First observe that
\begin{equation} \label{ratio}
\frac{\prob[\Theta_t X_t>x]}{\prob[X_1>x]}=\int_{0}^{\infty}\frac{\prob[X_1>x/v]}{\prob[X_1>x]}G_t(dv).
\end{equation}
We break the range of integration into three intervals $(0,1]$, $(1,x]$ and $(x,\infty)$, where we choose a suitably large $x$ greater than $1$.

Since $\overline F$ is regularly varying of index $-\alpha$ with $\alpha>0$, $\prob[X_1>x/v]/\prob[X_1>x]$ converges uniformly to $v^\alpha$ for $v\in(0,1)$ or equivalently $1/v\in(1, \infty)$. Hence the integral in~\eqref{ratio} over the first interval can be bounded, for all large enough $x$, as
\begin{equation} \label{eq: first bd}
\int_{0}^{1}\frac{\prob[X_1>x/v]}{\prob[X_1>x]}G_t(dv) \le 2\E[\Theta_t^\alpha].
\end{equation}

For the integral in~\eqref{ratio} over the third interval, we have, for all large enough $x$, by~\eqref{Ct} (for the conditions~\eqref{DZ2},~\eqref{DZ3} and~\eqref{DZ4} only),
\begin{multline} \label{eq: third bd}
\int_{x}^{\infty}\frac{\prob[X_1>x/v]}{\prob[X_1>x]}G_t(dv) \le \frac{\prob\left[\Theta_t>x\right]}{\prob[X_1>x]}\\
\le
\begin{cases}
\frac{\E\left[\Theta_t^\alpha\right]}{L(x)} \leq 2D_1 \frac{\E\left[\Theta_t^\alpha\right]}{L(1)}, &\text{by Markov's inequality, when~\eqref{DZ1} holds,}\\
C_t, &\text{when~\eqref{DZ2} holds,}\\
\frac{\prob\left[\Theta_t>x\right]}{c(x)x^{-\alpha}\prob[U>\log x]} \le \frac2c C_t, &\text{when~\eqref{DZ3} holds,}\\
C_t, &\text{as $m(x)\to\infty$, when~\eqref{DZ4} holds.}
\end{cases}
\end{multline}
Note that, when the condition~\eqref{DZ3} holds and $L$ is of type 4, we can ignore the factor $\prob[V>\log x]$, as it is bounded by $1$.

Finally, we consider the integral in~\eqref{ratio} over the second interval separately for each of the DZ conditions. We begin with the condition~\eqref{DZ1}. In this case, we have, for all large enough $x$,
\begin{multline} \label{eq: second bd DZ1}
\int_{1}^{x}\frac{\prob[X_1>x/v]}{\prob[X_1>x]}G_t(dv) \le \int_{1}^{x} v^\alpha \frac{L(x/v)}{L(x)} G_t(dv)\\ \le \sup_{y\in[1,x]} \frac{L(y)}{L(x)} \E\left[\Theta_t^\alpha\right] \le 2D_1 \E\left[\Theta_t^\alpha\right].
\end{multline}

Next we consider the condition~\eqref{DZ2}. Integrating by parts, we have
$$\int_{1}^{x}\frac{\prob\left[X_1>{x}/{v}\right]}{\prob[X_1>x]}G_t(dv) \leq \prob[\Theta_t > 1] + \int_{1}^{x} \frac{\prob[\Theta_t > v]}{\prob[X_1 > x]} d_{v}\prob\left[X_1 > {x}/{v}\right].$$ Using Markov's inequality and~\eqref{Ct} respectively in each of the terms, we have
$$\int_{1}^{x}\frac{\prob\left[X_1>{x}/{v}\right]}{\prob[X_1>x]}G_t(dv) \leq \E[\Theta_t^\alpha] + C_t \int_1^x \frac{\prob[X_1 > v]}{\prob[X_1 > x]} d_{v}\prob\left[X_1 > {x}/{v}\right].$$
Substituting $u=\log v$, the second term becomes, for all large $x$,
\begin{multline*}
C_t \int_0^{\log x} \frac{\prob[\log X_1 > u]}{\prob[\log X_1 > \log x]} d_u\prob\left[\log X_1 > \log x - u \right]\\
\le 2 C_t \E[\exp(\alpha (\log X_1)^+)] \le 2 C_t \E[X_1^\alpha],
\end{multline*}
where the inequalities follow, since $L(e^x)\in\mathcal S_d$ implies $(\log X_1)^+\in \mathcal S(\alpha)$, cf. \cite{kluppelberg:1989}. Thus,
\begin{equation} \label{eq: second bd DZ2}
\int_{1}^{x}\frac{\prob[X_1>x/v]}{\prob[X_1>x]}G_t(dv) \le \E[\Theta_t^\alpha] + 2 C_t \E[X_1^\alpha].
\end{equation}

Next we consider the condition~\eqref{DZ3}. In this case, we have
\begin{align*}
\int_1^{x}\frac{\prob[X_1>x/v]}{\prob[X_1>x]}G_t(dv)&= \int_1^{x}\frac{L(x/v)}{L(x)}v^{\alpha}G_t(dv)\\
&\leq \sup_{v\in[1,x]}\frac{c(x/v)}{c(x)} \int_1^{x}\frac{\prob[U > \log x -\log v]}{\prob[U > \log x]}v^{\alpha}G_t(dv).
\end{align*}
If $L$ is of type 4, the ratio $L(x/v)/L(x)$ has an extra factor $\prob[V>\log x]/\prob[V>\log x-\log v]$, which is bounded by $1$. Thus the above estimate works if $L$ is either of type 3 or of type 4. Since $c(x)\rightarrow c \in (0,\infty)$, we have $\sup_{v\in[N,x)} {c(x/v)}/{c(x)} := L_2 <\infty$. Integrating by parts, the integral becomes
\begin{multline*}
\int_1^{x}\frac{\prob[U>\log x-\log v]}{\prob[U>\log x]}v^{\alpha}G_t(dv)\\
\leq \prob[\Theta_t >1]+\int_1^{x}\frac{\prob[U>\log x- \log v]\prob[\Theta_t > v]}{\prob[U>\log x]}\alpha v^{\alpha -1}dv\\
+\int_1^{x}\frac{\prob[\Theta_t > v]v^{\alpha}}{\prob[U>\log x]}d_{v}\prob[U>\log x-\log v].
\end{multline*}
The first term is bounded by $\E[\Theta_1^\alpha]$ by Markov's inequality. By~\eqref{Ct}, the second term gets bounded by, for all large enough $x$,
$$\alpha C_t \int_1^x \frac{\prob[U>\log x-\log v]\prob[U>\log v]}{\prob[U>\log x]}d(\log v) \le 2 \alpha C_t \E[U],$$
as $U$ belongs to $\mathcal S^*$. Again, by~\eqref{Ct}, the third term gets bounded by, for all large enough $x$,
$$C_t \int_1^x \frac{\prob[U>\log v] d_v \prob[U>\log x - \log v]}{\prob[U>\log x]} \le 4 C_t,$$
as $U$ belongs to $\mathcal S^*$ and hence is subexponential, cf. \cite{kluppelberg:1988}. Combining the bounds for the three terms, we get
\begin{equation} \label{eq: second bd DZ3}
\int_{1}^{x}\frac{\prob[X_1>x/v]}{\prob[X_1>x]}G_t(dv) \le L_2 \{\E[\Theta_t^\alpha] + 2 (\alpha \E[U] + 2) C_t\}.
\end{equation}

Finally we consider the condition~\eqref{DZ4}. In this case, we split the interval $(1,x]$ into two subintervals $(1,\sqrt{x}]$ and $(\sqrt{x},x]$ and bound the integrals on each of the subintervals separately. We begin with the integral on the subinterval $(1,\sqrt{x}]$.
$$\int_1^{\sqrt x} \frac{L(x/v)}{L(x)}v^{\alpha}G_t(dv)\leq \sup_{v\in (1,\sqrt x]}\frac{L(x/v)}{L(x)}\int_1^{\sqrt x}v^{\alpha}G_t(dv) \leq D_2 \E[\Theta_t^{\alpha}].$$
For the integral over $(\sqrt{x}, x]$, we integrate by parts to obtain
$$\int_{\sqrt x}^{x}\frac{L(x/v)}{L(x)}v^{\alpha}G_t(dv)\leq \prob[\Theta_t>\sqrt x]x^{\alpha/2}\frac{L(\sqrt x)}{L(x)}+\int_{\sqrt x}^x\frac{\prob[\Theta_t>v]}{L(x)}d_v(v^{\alpha}L(x/v)).$$
By Markov's inequality, the first term is bounded by $D_2 \E[\Theta_t^\alpha]$. The second term becomes, using~\eqref{Ct},
\begin{align*}
\int_{\sqrt x}^x \frac{\prob[\Theta_t>v]}{L(x)} x^\alpha d_v(\prob[X_1\le x/v]) &\leq C_t\int_{\sqrt x}^x \frac{\prob[X_1>v]}{L(x)m(v)} x^\alpha d_v(\prob[X_1\le x/v])\\
&\leq \frac{C_t}{m(\sqrt x)}\int_{\sqrt x}^x\frac{L(v)}{L(x)}\left(\frac{x}{v}\right)^{\alpha}d_v(\prob[X_1\leq x/v])\\
&\leq \frac{D_2 C_t}{m(\sqrt x)} \int_{1}^{\sqrt x}y^{\alpha}d_y(\prob[X_1\leq y]) \le D_2 C_t.
\end{align*}
Combining the bounds for the integrals over each subinterval, we get
\begin{equation} \label{eq: second bd DZ4}
\int_{1}^{x}\frac{\prob[X_1>x/v]}{\prob[X_1>x]}G_t(dv) \le D_2 (2 \E[\Theta_t^\alpha] + C_t).
\end{equation}

Combining all the bounds in~\eqref{eq: first bd}--\eqref{eq: second bd DZ4}, for some constant $B$, we can choose the bound in~\eqref{eq: dir bd} as
$$B_t =
\begin{cases}
B \E[\Theta_t^\alpha], &\text{when the condition~\eqref{DZ1} holds,}\\
B (\E[\Theta_t^\alpha] + C_t), &\text{when the conditions~\eqref{DZ2}, \eqref{DZ3} or~\eqref{DZ4} hold.}
\end{cases}$$
Then, for $\alpha<1$, the summability condition~\eqref{eq: bd sum} follows from the condition~\eqref{ERW1} alone under the condition~\eqref{DZ1} and from the condition~\eqref{ERW1} together with~\eqref{sum less} under the conditions~\eqref{DZ2},~\eqref{DZ3} or~\eqref{DZ4}. For $\alpha\ge 1$, under the condition~\eqref{DZ1}, the summability condition~\eqref{eq: bd sum} follows from the condition~\eqref{ERW2}. Finally, to check the summability condition~\eqref{eq: bd sum} for $\alpha\ge1$, under the condition~\eqref{DZ2},~\eqref{DZ3} or~\eqref{DZ4}, observe that as $\alpha\ge 1$ and $\epsilon>0$, we have
$$(\E[\Theta_t^\alpha] + C_t)^{\frac1{\alpha+\epsilon}} \le (\E[\Theta_t^\alpha])^{\frac1{\alpha+\epsilon}} + C_t^{\frac1{\alpha+\epsilon}}$$ and we get the desired condition from the condition~\eqref{ERW2}, together with~\eqref{sum more}.
\end{proof}
\end{section}

\begin{section}{The tails of the summands from the tail of the sum} \label{sec: converse}
In this section, we address the converse problem of studying the tail behavior of $X_1$ based on the tail behavior of $\xinf$. For the converse problem, we restrict ourselves to the setup where the sequence $\{X_t\}$ is positive and pairwise asymptotically independent and the other sequence $\{\Theta_t\}$ is positive and independent of the sequence $\{X_t\}$, such that $\xinf$ is finite with probability one and has regularly varying tail of index $-\alpha$. Depending on the value of $\alpha$, we assume the usual RW moment conditions~\eqref{RW1} or~\eqref{RW2} for the sequence $\{\Theta_t\}$, instead of the modified ones. Then, under a further assumption of the non-vanishing Mellin transform along the vertical line of the complex plane with the real part $\alpha$, we shall show that $X_1$ also has regularly varying tail of index $-\alpha$.

We use the extension of the notion of product of two independent positive random variables to the product convolution of two measures on $(0,\infty)$, which we allow to be $\sigma$-finite. For two $\sigma$-finite measures $\nu$ and $\rho$ on $(0,\infty)$, we define the product convolution as
$$\nu \circledast \rho(B)= \int_0^{\infty}\nu(x^{-1}B)\rho(dx),$$
for any Borel subset $B$ of $(0,\infty)$. We shall need the following result from \cite{JMRS:2009}.
\begin{theorem}[\citealp{JMRS:2009}, Theorem~2.3] \label{JMRS}
Let a non-zero $\sigma$-finite measure $\rho$ on $(0,\infty)$ satisfies, for some $\alpha>0$, $\epsilon\in(0,\alpha)$ and all $\beta\in\mathbb R$,
\begin{align}
\int_{0}^{\infty} \left(y^{\alpha-\epsilon}\vee y^{\alpha+\epsilon}\right) \rho(dy)&<\infty \label{conditionjmrs0}
\intertext{and}
\int_{0}^{\infty}y^{\alpha+i\beta}\rho(dy)&\neq 0. \label{conditionjmrs1}
\end{align}

Suppose, for another $\sigma$-finite measure $\nu$ on $(0,\infty)$, the product convolution measure $\nu\circledast\rho$ has a regularly varying tail of index $-\alpha$ and
\begin{equation} \label{conditionjmrs} \lim_{b\rightarrow0}\limsup_{x\rightarrow\infty}\frac{\int_0^b\rho(x/y,\infty)\nu(dy)}{(\nu\circledast\rho)(x,\infty)}=0.
\end{equation}
Then the measure $\nu$ has a regularly varying tail of index $-\alpha$ as well and
$$\lim_{x\rightarrow\infty}\frac{\nu\circledast\rho (x,\infty)}{\nu(x,\infty)}=\int_0^{\infty}y^{\alpha}\rho(dy).$$

Conversely, if~\eqref{conditionjmrs0} holds but~\eqref{conditionjmrs1} fails for the measure $\rho$, then there exists a $\sigma$-finite measure $\nu$ without regularly varying tail, such that $\nu\circledast\rho$ has regularly varying tail of index $-\alpha$ and~\eqref{conditionjmrs} holds.
\end{theorem}

\begin{remark} \label{rem: jmrs}
\cite{JMRS:2009} gave an explicit construction of the $\sigma$-finite measure $\nu$ in Theorem~\ref{JMRS} above. In fact, if~\eqref{conditionjmrs1} fails for $\beta=\beta_0$, then, for any real number $a$ and $b$ satisfying $0<a^2+b^2\le 1$, we can define $g(x) = 1 + a \cos(\beta_0 \log x) + b \sin(\beta_0 \log x)$ and $d\nu = g d\nu_\alpha$ will qualify for the measure in the converse part, where $\nu_\alpha$ is the $\sigma$-finite measure given by $\nu_\alpha(x,\infty)=x^{-\alpha}$ for any $x>0$.

It is easy to check that $0\le g(x)\le2$ for all $x>0$ and hence
\begin{equation} \label{eq: bd nu}
\nu(x,\infty)\le2x^{-\alpha}.
\end{equation}
Also, it is known from Theorem~2.1 of \cite{JMRS:2009} that
\begin{equation} \label{eq: prod conv}
\nu\circledast\rho=\|\rho\|_\alpha\nu_\alpha,
\end{equation}
where $\|\rho\|_\alpha = \int_0^\infty y^\alpha \rho(dy) <\infty$, by~\eqref{conditionjmrs0}.
\end{remark}

We are now ready to state the main result of this section.
\begin{theorem}\label{main result}
Let $\{X_t,t\geq1\}$ be a sequence of identically distributed, pairwise asymptotically independent positive random variables and $\{\Theta_t,t\geq1\}$ be a sequence of positive random variables independent of $\{X_t\}$, such that $\xinf=\sum_{t=1}^{\infty}\Theta_tX_t$ is finite with probability one and has regularly varying tail of index $-\alpha$, where $\alpha>0$. Let $\{\Theta_t,t\geq1\}$ satisfy the appropriate RW condition~\eqref{RW1} or~\eqref{RW2}, depending on the value of $\alpha$. If we further have, for all $\beta\in \mathbb{R}$,
\begin{equation} \label{eq: mellin}
\sum_{t=1}^{\infty}\E[\Theta_t^{\alpha+i\beta}]\neq 0,
\end{equation}
then $X_1$ has regularly varying tail of index $-\alpha$ and, as $x\to\infty$,
$$\prob[\xinf > x]\sim \prob[X_1>x]\sum_{t=1}^{\infty}\E[\Theta_t^{\alpha}] \ \text{as}\ x\rightarrow\infty.$$
\end{theorem}

We shall prove Theorem~\ref{main result} in several steps. We collect the preliminary steps, which will also be useful for a converse to Theorem~\ref{main result}, into three separate lemmas. The first lemma controls the tail of the sum $\xinf$.
\begin{lemma} \label{lem: tail}
Let $\{X_t\}$ be a sequence of identically distributed positive random variables and $\{\Theta_t\}$ be a sequence of positive random variables independent of $\{X_t\}$. Suppose that the tail of $X_1$ is dominated by a bounded regularly varying function $R$ of index $-\alpha$, where $\alpha>0$, that is, for all $x>0$,
\begin{equation} \label{eq: domination}
\prob[X_1>x] \le R(x).
\end{equation}
Also assume that $\{\Theta_t\}$ satisfies the appropriate RW condition depending on the value of $\alpha$. Then,
\begin{align*}
\lim_{m\to\infty} \limsup_{x\to\infty} \frac{\prob[\xupk>x]}{R(x)} &= 0 \intertext{and}
\lim_{m\to\infty} \limsup_{x\to\infty} \sum_{t=m+1}^\infty \frac{\prob[\Theta_t X_t>x]}{R(x)} &= 0.
\end{align*}
\end{lemma}
\begin{proof}
From~\eqref{two terms tail negligible}, we have
\begin{equation} \label{eq: tail bd conv}
\prob\left[\xupk>x\right] \le \sum_{t=m+1}^\infty \prob\left[\Theta_tX_t>x\right] + \prob\left[\sum_{t=m+1}^\infty \Theta_tX_t \mathbbm 1_{[\Theta_tX_t\le x]}>x\right].
\end{equation}
Using~\eqref{eq: domination}, the summands of the first term on the right side of~\eqref{eq: tail bd conv} can be bounded as
\begin{equation} \label{eq: first term conv}
\prob[\Theta_t X_t > x] = \int_0^\infty \prob[X_t > x/u] G_t(du) \leq \int_0^\infty R(x/u) G_t(du).
\end{equation}

Before analyzing the second term on the right side of~\eqref{eq: tail bd conv}, observe that, for $\gamma>\alpha$, we have, using Fubini's theorem,~\eqref{eq: domination} and Karamata's theorem successively
$$\E\left[X_t^\gamma \mathbbm 1_{[X_t\le x]}\right] \le \gamma \int_0^x u^{\gamma-1} \prob[X_t>u] du \le \gamma \int_0^x u^{\gamma-1} R(u) du \sim \frac{\gamma}{\gamma-\alpha} x^\gamma R(x).$$
Thus, there exists constant $M \equiv M(\gamma)$, such that, for all $x>0$,
\begin{equation} \label{eq: karamata}
x^{-\gamma} \E\left[X_t^\gamma \mathbbm 1_{[X_t\le x]}\right] \le M R(x).
\end{equation}

We bound the second term on the right side of~\eqref{eq: tail bd conv}, using~\eqref{eq: karamata}, separately for the cases $\alpha<1$ and $\alpha\ge 1$. For $\alpha<1$, we use~\eqref{eq: term 2 alpha small} and~\eqref{eq: karamata} with $\gamma=1$, to get
\begin{equation} \label{eq: second term alpha small}
\prob\left[\sum_{t=m+1}^\infty \Theta_tX_t \mathbbm 1_{[\Theta_tX_t\le x]}>x\right] \le M(1) \sum_{t=m+1}^\infty \int_0^\infty R(x/u) G_t(du).
\end{equation}
For $\alpha\ge 1$, we use~\eqref{eq: term 2 alpha large} and~\eqref{eq: karamata} with $\gamma=\alpha+\epsilon$, to get
\begin{equation} \label{eq: second term alpha large}
\prob\left[\sum_{t=m+1}^\infty \Theta_tX_t \mathbbm 1_{[\Theta_tX_t\le x]}>x\right] \le M(\alpha+\epsilon) \left[ \sum_{t=m+1}^\infty \left( \int_0^\infty R(x/u) G_t(du) \right)^{\frac1{\alpha+\epsilon}} \right]^{\alpha+\epsilon}.
\end{equation}

Combining~\eqref{eq: first term conv},~\eqref{eq: second term alpha small} and~\eqref{eq: second term alpha large} with the bound in~\eqref{eq: tail bd conv}, the proof will be complete if we show
\begin{equation} \label{eq: aim}
\begin{split}
&\lim_{m\to\infty} \limsup_{x\to\infty} \sum_{t=m+1}^\infty \int_0^\infty \frac{R(x/u)}{R(x)} G_t(du) = 0, \text{ for $\alpha<1$},\\
\text{and } &\lim_{m\to\infty} \limsup_{x\to\infty} \sum_{t=m+1}^\infty \left( \int_0^\infty \frac{R(x/u)}{R(x)} G_t(du) \right)^{\frac1{\alpha+\epsilon}} =0, \text{ for $\alpha\ge1$}.
\end{split}
\end{equation}
Note that, for $\alpha\ge 1$, as in Remark~\ref{rem: RW}, the second limit above gives the first one as well.

We bound the integrand using a variant of Potter's bound  \citep[see][Lemma 2.2 ]{resnick:willekens:1991}. Let $\epsilon>0$ be as in the RW conditions. Then there exists a $x_0$ and a constant $M>0$ such that, for $x>x_0$, we have
\begin{equation} \label{eq: potter}
\frac{R(x/u)}{R(x)}\leq
\begin{cases}
Mu^{\alpha-\epsilon},  &\text{if $u<1$,}\\
Mu^{\alpha +\epsilon},  &\text{if $1\leq u\leq x/x_0$.}
\end{cases}
\end{equation}

We split the range of integration in~\eqref{eq: aim} into three intervals, namely $(0,1]$, $(1,x/x_0]$ and $(x/x_0,\infty)$. For $x>x_0$, we bound the integrand over the first two integrals using~\eqref{eq: potter} and hence the integrals get bounded by a multiple of $\E[\Theta_t^{\alpha-\epsilon}]$ and $\E[\Theta_t^{\alpha+\epsilon}]$ respectively. As $R$ is bounded, by Markov's inequality, the third integral gets bounded by a multiple of $x_0^{\alpha+\epsilon}\E[\Theta_t^{\alpha+\epsilon}]/\{x^{\alpha+\epsilon}R(x)\}$. Putting all the bounds together, we have
$$\int_0^\infty \frac{R(x/u)}{R(x)} G_t(du) \le M \left(\E[\Theta_t^{\alpha-\epsilon}] + \E[\Theta_t^{\alpha+\epsilon}] + \frac{x_0^{\alpha+\epsilon}\E[\Theta_t^{\alpha+\epsilon}]}{x^{\alpha+\epsilon}R(x)}\right).$$
Then,~\eqref{eq: aim} holds for $\alpha<1$ using the condition~\eqref{RW1} and the fact that $R$ is regularly varying of index $-\alpha$. For $\alpha\ge 1$, we need to further observe that, as $\alpha+\epsilon>1$, we have
\begin{multline*}
\left(\int_0^\infty \frac{R(x/u)}{R(x)} G_t(du)\right)^{\frac1{\alpha+\epsilon}} \le M^{\frac1{\alpha+\epsilon}} \left[ \left( \E[\Theta_t^{\alpha-\epsilon}] + \E[\Theta_t^{\alpha+\epsilon}]\right) + \frac{x_0^{\alpha+\epsilon} \E[\Theta_t^{\alpha+\epsilon}]}{x^{\alpha+\epsilon}R(x)}\right]^{\frac1{\alpha+\epsilon}}\\
\le M^{\frac1{\alpha+\epsilon}} \left(\E[\Theta_t^{\alpha-\epsilon}] + \E[\Theta_t^{\alpha+\epsilon}]\right)^{\frac1{\alpha+\epsilon}} + \frac{x_0 \left(\E[\Theta_t^{\alpha+\epsilon}]\right)^{\frac1{\alpha+\epsilon}}}{xR(x)^{\frac1{\alpha+\epsilon}}}
\end{multline*}
and~\eqref{eq: aim} holds using the condition~\eqref{RW2} and the fact that $R$ is regularly varying of index $-\alpha$.
\end{proof}

The next lemma considers the joint distribution of $(\Theta_1 X_1, \Theta_2 X_2)$ and shows they are ``somewhat'' asymptotically independent, if $(X_1, X_2)$ are asymptotically independent.
\begin{lemma} \label{lem: approx indep}
Let $(X_1, X_2)$ and $(\Theta_1, \Theta_2)$ be two independent random vectors, such that each coordinate of either vector is positive. We assume that $X_1$ and $X_2$ have same distribution with their common tail dominated by a regularly varying function $R$ of index $-\alpha$ with $\alpha>0$, as in~\eqref{eq: domination}. We also assume that $R$ stays bounded away from $0$ on any bounded interval. We further assume that both $\Theta_1$ and $\Theta_2$ have $(\alpha+\epsilon)$-th moments finite. Then
$$\lim_{x\to\infty} \frac{\prob[\Theta_1 X_1 > x, \Theta_2 X_2 > x]}{R(x)} = 0.$$
\end{lemma}
\begin{proof}
By asymptotic independence and~\eqref{eq: domination}, we have
\begin{equation} \label{eq: jt neg}
\prob[X_1>x,X_2>x] = \lito(R(x)).
\end{equation}
Further, since $R$ is bounded away from $0$ on any bounded interval, $\prob[X_1>x,X_2>x]$ is bounded by a multiple of $R(x)$. Then,
\begin{align*}
\frac{\prob[\Theta_1X_1>x,\Theta_2X_2>x]}{R(x)} =&\int\limits_0^{\infty}\int\limits_0^{\infty}\frac{\prob[X_1>x/u,X_2>x/v]}{R(x)}G(du,dv)\\
=&\iint\limits_{u>v}+\iint\limits_{u\leq v}\frac{\prob[X_1>x/u,X_2>x/v]}{R(x)}G(du,dv)\\
\leq &\int_0^{\infty} \frac{\prob[X_1>x/u,X_2>x/u]}{R(x)} (G_1+G_2)(du)\\
= &\int_0^\infty \frac{\prob[X_1>x/u,X_2>x/u]}{R(x)} \mathbbm 1_{[0,x/x_0]}(u) (G_1+G_2)(du)\\ &\quad + \frac{x_0^{\alpha+\epsilon} \left( \E[\Theta_1^{\alpha+\epsilon}] + \E[\Theta_1^{\alpha+\epsilon}] \right)}{x^{\alpha+\epsilon} R(x)},
\end{align*}
for any $x_0>0$. The integrand in the first term goes to $0$, using~\eqref{eq: jt neg} and the regular variation of $R$. Further choose $x_0$ as in Potter's bound~\eqref{eq: potter}. Then, the integrand of the first term is bounded by a multiple of $1+u^{\alpha+\epsilon}$, which is integrable with respect to $G_1+G_2$. So, by Dominated Convergence Theorem, the first term goes to $0$. For this choice of $x_0$, the second term also goes to $0$, as $R$ is regularly varying of index $-\alpha$.
\end{proof}

The next lemma compares $\sum_{t=1}^m \prob[\Theta_t X_t > x]$ and $\prob\left[ \sum_{t=1}^m \Theta_t X_t > x \right]$.
\begin{lemma} \label{lem: comp}
Let $\{X_t\}$ and $\{\Theta_t\}$ be two sequences of positive random variables.
Then, we have, for any $\frac12<\delta<1$ and $m\ge 2$,
\begin{align}
\prob\left[ \sum_{t=1}^m \Theta_t X_t > x \right] \geq &\sum_{t=1}^m \prob[\Theta_t X_t > x] - \underset{1\le s\neq t\le m}{\sum\sum} \prob[\Theta_s X_s > x, \Theta_t X_t > x] \label{eq: comp up}
\intertext{and}
\prob\left[ \sum_{t=1}^m \Theta_t X_t > x \right] \leq &\sum_{t=1}^m \prob[\Theta_t X_t > x] \nonumber\\
&\quad +\underset{1\le s\neq t\le m}{\sum\sum} \prob\left[\Theta_s X_s > \frac{1-\delta}{m-1} x, \Theta_t X_t > \frac{1-\delta}{m-1} x\right]. \label{eq: comp dn}
\end{align}
\end{lemma}
\begin{proof}
The first inequality~\eqref{eq: comp up} follows from the fact that
$$\left[\sum_{t=1}^m \Theta_t X_t > x\right] \subseteq \bigcup_{t=1}^m [\Theta_t X_t > x]$$ and Bonferroni's inequality.

For the second inequality~\eqref{eq: comp dn}, observe that
$$\prob\left[\xdnk>x\right] \leq \sum_{t=1}^{m}\prob[\xtheta>\delta x]+\prob\left[\sum_{t=1}^{k}\xtheta>x, \bigvee_{t=1}^m\xtheta \leq \delta x\right].$$
Next we estimate the second term as
\begin{align*}
&\prob\left[\sum_{t=1}^{m} \xtheta>x, \bigvee_{t=1}^m \xtheta \leq \delta x\right]\\
=& \prob\left[\sum_{t=1}^m \xtheta>x, \bigvee_{t=1}^m\xtheta \leq \delta x, \bigvee_{t=1}^m \xtheta>\frac xm\right]\\
\leq &\sum_{s=1}^m \prob\left[\sum_{t=1}^{m}\xtheta>x, \bigvee_{t=1}^m \xtheta \leq \delta x, \Theta_{s}X_s>\frac xm \right]\\
\leq &\sum_{s=1}^m \prob\left[\sum_{t=1}^{m}\xtheta>x, \Theta_sX_s \leq \delta x, \Theta_sX_s>\frac xm \right]\\
\leq &\sum_{s=1}^m \prob\left[\sum_{\substack{t=1\\t\neq s}}^k \xtheta>(1-\delta)x, \Theta_sX_s>\frac xm \right]\\
\leq &\underset{1\le s\neq t\le m}{\sum\sum} \prob\left[\xtheta>\frac{1-\delta}{m-1}x, \Theta_sX_s>\frac xm \right]\\
\leq &\underset{1\le s\neq t\le m}{\sum\sum} \prob\left[\xtheta>\frac{1-\delta}{m-1}x, \Theta_sX_s>\frac{1-\delta}{m-1}x \right],
\end{align*}
since $\delta>1/2$ and $m\ge2$ imply $(1-\delta)/(m-1)<1/m$.
\end{proof}

With the above three lemmas, we are now ready to show the tail equivalence of the distribution of $\xinf$ and $\sum_{t=1}^\infty \prob[\Theta_tX_t\in\cdot]$.
\begin{proposition} \label{prop: equiv}
Let $\{X_t,t\geq1\}$ be a sequence of identically distributed, pairwise asymptotically independent positive random variables and $\{\Theta_t,t\geq1\}$ be a sequence of positive random variables independent of $\{X_t\}$, such that $\xinf=\sum_{t=1}^{\infty}\Theta_tX_t$ is finite with probability one and has regularly varying tail of index $-\alpha$, where $\alpha>0$. Let $\{\Theta_t,t\geq1\}$ satisfy the appropriate RW condition~\eqref{RW1} or~\eqref{RW2}, depending on the value of $\alpha$. Then, as $x\to\infty$,
$$\sum_{t=1}^\infty \prob[\Theta_tX_t>x] \sim \prob[\xinf>x].$$
\end{proposition}
\begin{proof}
We first show that the tail of $X_1$ can be dominated by a multiple of the tail of $\xinf$, so that Lemmas~\ref{lem: tail} and~\ref{lem: approx indep} apply. Note that the tail of $\xinf$ is bounded and stays bounded away from $0$ on any bounded interval. As $\Theta_1$ is a positive random variable, choose $\eta>0$ such that $\prob[\Theta_1>\eta]>0$. Then, for all $x>0$,
$$\prob[\xinf>\eta x] \ge \prob[\Theta_1X_1>\eta x, \Theta_1>\eta] \ge \prob[X_1>x] \prob[\Theta_1>\eta].$$
Further, using the regular variation of the tail of $\xinf$, $X_1$ satisfies~\eqref{eq: domination} with $R$ as a multiple of $\prob[\xinf>\cdot]$. Thus, from Lemmas~\ref{lem: tail} and~\ref{lem: approx indep}, we have,
\begin{align}
\lim_{m\to\infty} \limsup_{x\to\infty} \frac{\prob[\xupk>x]}{\prob[\xinf>x]} &= 0, \label{eq: neg tail one}\\
\lim_{m\to\infty} \limsup_{x\to\infty} \sum_{t=m+1}^\infty \frac{\prob[\xtheta>x]}{\prob[\xinf>x]} &= 0, \label{eq: neg tail two}
\intertext{and, for any $s\neq t$,}
\lim_{x\to\infty} \frac{\prob[\Theta_s X_s > x, \xtheta > x]}{\prob[\xinf > x]} &= 0. \label{eq: approx indep}
\end{align}

Choose any $\delta>0$. Then
$$\prob\left[\xinf>(1+\delta)x\right]\leq\prob\left[\xdnk>x\right]+\prob\left[\xupk>\delta x\right],$$
and from~\eqref{eq: neg tail one} and the regular variation of the tail of $\xinf$, we have
$$\lim_{m\to\infty}\liminf_{x\to\infty}\frac{\prob[\xdnk>x]}{\prob[\xinf>x]}\geq 1.$$
Further, using the trivial bound $\prob[\xdnk>x] \le \prob[\xinf>x]$, we have
\begin{equation}\label{eq: lower tail}
1 \leq \lim_{m\to\infty} \liminf_{x\to\infty} \frac{\prob[\xdnk>x]}{\prob[\xinf>x]} \leq \lim_{m\to\infty} \limsup_{x\to\infty} \frac{\prob[\xdnk>x]}{\prob[\xinf>x]} \leq 1.
\end{equation}

We next replace $\prob[\xdnk>x]$ in the numerator by $\sum_{t=1}^m \prob[\xtheta>x]$. We obtain the upper bound first. From~\eqref{eq: comp up},~\eqref{eq: approx indep} and~\eqref{eq: lower tail}, we get
$$\limsup_{x\to\infty} \frac{\sum_{t=1}^m \prob[\xtheta>x]}{\prob[\xinf>x]} \leq 1$$ and letting $m\to\infty$, we get the upper bound. The lower bound follows using exactly similar lines, but using~\eqref{eq: comp dn} and the regular variation of the tail of $\xinf$ instead of~\eqref{eq: comp up}. Putting together, we get
\begin{equation}\label{eq: lower tail sum}
1 \leq \lim_{m\to\infty} \liminf_{x\to\infty} \frac{\sum_{t=1}^m \prob[\xtheta>x]}{\prob[\xinf>x]} \leq \lim_{m\to\infty} \limsup_{x\to\infty} \frac{\sum_{t=1}^m \prob[\xtheta>x]}{\prob[\xinf>x]} \leq 1.
\end{equation}

Then the result follows combining~\eqref{eq: neg tail two} and~\eqref{eq: lower tail sum}.
\end{proof}

We are now ready to prove Theorem~\ref{main result}.
\begin{proof}[Proof of Theorem~\ref{main result}]
Let $\nu$ be the law of $X_1$ and define the measure $\rho(\cdot)=\sum_{t=1}^{\infty}\prob[\Theta_t \in \cdot]$. As observed in Remark~\ref{rem: RW}, under the RW conditions, for all values of $\alpha$, we have $\sum_{t=1}^\infty \E[\Theta_t^{\alpha+\epsilon}] <\infty$. Thus, $\rho$ is a $\sigma$-finite measure. Also, by Proposition~\ref{prop: equiv}, we have $\nu\circledast\rho (x,\infty) = \sum_{t=1}^{\infty} \prob[\xtheta>x] \sim \prob[\xinf>x]$. Hence $\nu\circledast\rho$ has regularly varying tail of index $-\alpha$. As $\nu$ is a probability measure, by Remark~2.4 of \cite{JMRS:2009},~\eqref{conditionjmrs} holds. The RW condition implies~\eqref{conditionjmrs0}. Finally,~\eqref{conditionjmrs1} holds, since, for all $\beta\in\mathbb{R}$, we have, from~\eqref{eq: mellin}, $\int_0^{\infty} y^{\alpha+i\beta} \rho(dy) = \sum_{t=1}^{\infty} \E[\Theta_t^{\alpha+i\beta}] \neq 0$. Hence, by Theorem~\ref{JMRS}, $X_1$ has regularly varying tail of index $-\alpha$.
\end{proof}

As in Theorem~\ref{JMRS},~\eqref{eq: mellin} is necessary for Theorem~\ref{main result} and we give its converse below.
\begin{theorem} \label{thm: converse}
Let $\{\Theta_t,t\geq1\}$ be a sequence of positive random variables satisfying the condition~\eqref{RW1} or~\eqref{RW2}, for some $\alpha>0$, but $\sum_{t=1}^{\infty}\E[\Theta_t^{\alpha+i\beta_0}]=0$ for some $\beta_0\in\mathbb{R}$. Then there exists a sequence of i.i.d.\ positive random variables $\{X_t\}$, such that $X_1$ does not have a regularly varying tail, but $\xinf$ is finite almost surely and has regularly varying tail of index $-\alpha$.
\end{theorem}

The proof depends on an analogue of Proposition~\ref{prop: equiv}.
\begin{proposition} \label{prop: equiv alt}
Let $\{X_t,t\geq1\}$ be a sequence of identically distributed, pairwise asymptotically independent positive random variables and $\{\Theta_t,t\geq1\}$ be a sequence of positive random variables satisfying the condition~\eqref{RW1} or~\eqref{RW2} for some $\alpha>0$ and independent of $\{X_t\}$. If $\sum_{t=1}^\infty \prob[\Theta_tX_t>x]$ is regularly varying of index $-\alpha$, then, as $x\to\infty$,
$$\sum_{t=1}^\infty \prob[\Theta_tX_t>x] \sim \prob[\xinf>x]$$
and $\xinf$ is finite with probability one.
\end{proposition}
\begin{proof}
We shall denote $R(x) = \sum_{t=1}^\infty \prob[\Theta_t X_t > x]$. As $\Theta_1$ is a positive random variable, choose $\eta>0$ such that $\prob[\Theta_1>\eta]>0$. Then, for all $x>0$, we have $R(x) \ge \prob[\Theta_1X_1>\eta x, \Theta_1>\eta] \ge \prob[X_1>x]\prob[\Theta_1>\eta]$
and using the regular variation of $R$, the tail of $X_1$ is dominated by a constant multiple of $R$. Also, note that, $R$ is bounded and stays bounded away from $0$ on any bounded interval. Then, from Lemmas~\ref{lem: tail} and~\ref{lem: approx indep}, we have
\begin{align}
\lim_{m\to\infty} \limsup_{x\to\infty} \frac{\prob[\xupk>x]}{R(x)} &= 0, \label{eq: neg tail one alt}\\
\lim_{m\to\infty} \limsup_{x\to\infty} \sum_{t=m+1}^\infty \frac{\prob[\xtheta>x]}{R(x)} &= 0, \label{eq: neg tail two alt}
\intertext{and, for any $s\neq t$,}
\lim_{x\to\infty} \frac{\prob[\Theta_s X_s > x, \xtheta > x]}{R(x)} &= 0. \label{eq: approx indep alt}
\end{align}

Using~\eqref{eq: neg tail two alt}, we have
$$1 \le \lim_{m\to\infty} \liminf_{x\to\infty} \frac{\sum_{t=1}^m \prob[\xtheta>x]}{R(x)} \le \lim_{m\to\infty} \limsup_{x\to\infty} \frac{\sum_{t=1}^m \prob[\xtheta>x]}{R(x)} \le 1.$$
As in the proof of Proposition~\ref{prop: equiv}, using~\eqref{eq: comp up},~\eqref{eq: comp dn} and~\eqref{eq: approx indep alt}, the above inequalities reduce to
$$1 \le \lim_{m\to\infty} \liminf_{x\to\infty} \frac{\prob[\xdnk>x]}{R(x)} \le \lim_{m\to\infty} \limsup_{x\to\infty} \frac{\prob[\xdnk>x]}{R(x)} \le 1$$
and the tail equivalence follows using~\eqref{eq: neg tail one alt} and the regular variation of $R$. Since $R(x)\to 0$, the tail equivalence also shows the almost sure finiteness of $\xinf$.
\end{proof}

Next, we prove Theorem~\ref{thm: converse} using the converse part of Theorem~\ref{JMRS}.
\begin{proof}[Proof of Theorem~\ref{thm: converse}]
Define the measure $\rho(\cdot)=\sum_{t=1}^{\infty}\prob[\Theta_t\in\cdot]$. By the RW moment condition, the measure $\rho$ is $\sigma$-finite. Further, we have, $\int_0^\infty y^{\alpha+i\beta_0} \rho(dy)=0$. Now by converse part of Theorem~\ref{JMRS}, there exists a $\sigma$-finite measure $\nu$, whose tail is not regularly varying, but $\nu\circledast\rho$ has regularly varying tail. Next, define a probability measure $\mu$ using the $\sigma$-finite measure $\nu$ as in Theorem~3.1 of~\cite{JMRS:2009}. Choose $b>1$, such that $\nu(b,\infty)\leq 1$ and define a probability measure on $(0,\infty)$ by
$$\mu(B)=\nu(B\cap(b,\infty))+(1-\nu(b,\infty))\mathbbm{1}_{B}(1),  \text{ where $B$ is Borel subset of $(0,\infty)$.}$$

First observe that
$$\mu(y,\infty) =
\begin{cases}
\nu(y,\infty), &\text{for $y>b$,}\\
\nu(b,\infty), &\text{for $1<y\le b$,}\\
1, &\text{for $y\le 1$.}
\end{cases}$$
Thus, $\mu$ does not have a regularly varying tail and
\begin{align*}
\mu\circledast\rho(x,\infty) = &\int_0^\infty \mu(x/u,\infty) \rho(du)\\
= &\int_0^{x/b} \nu(x/u,\infty) \rho(du) + \nu(b,\infty) \rho[x/b,x) + \rho[x,\infty)\\
= &\nu\circledast\rho(x,\infty) - 2 x^{-\alpha} \int_{x/b}^\infty u^\alpha \rho(du)\\
&\quad + \nu(b,\infty) \rho[x/b,x) + \rho[x,\infty).
\end{align*}
Now, using the bound from~\eqref{eq: bd nu} and~\eqref{eq: prod conv}, the second term is bounded by, for $x>b$,
$$2 \, \frac{\nu\circledast\rho(x,\infty)}{\|\rho\|_\alpha} \int_{x/b}^\infty u^{\alpha+\epsilon} \rho(du) = \lito(\nu\circledast\rho(x,\infty))$$
as $x\to\infty$, since $\int_0^\infty u^{\alpha+\epsilon} \rho(du) < \infty$ by the RW conditions. The sum of the last two terms can be bounded by
$$\frac{1+\nu(b,\infty)b^{\alpha+\epsilon}}{x^{\alpha+\epsilon}} \int_0^\infty u^{\alpha+\epsilon} \rho(du) = \lito(\nu\circledast\rho(x,\infty)),$$
as $x\to\infty$, since $\nu\circledast\rho(x,\infty)$ is regularly varying of index $-\alpha$. Thus, $\mu\circledast\rho(x,\infty) \sim \nu\circledast\rho(x,\infty)$ as $x\to\infty$ and hence is regularly varying of index $-\alpha$.

Let $X_t$ be an i.i.d.\ sequence with common law $\mu$. Then, $X_1$ does not have regularly varying tail. Further, by Proposition~\ref{prop: equiv alt}, $\xinf$ is finite with probability one and $\prob[\xinf>x] \sim \mu\circledast\rho(x,\infty)$ is regularly varying of index $-\alpha$.
\end{proof}
\end{section}

\end{document}